\documentclass[preprint,12pt]{elsarticle}

\usepackage{lmodern}
\usepackage{amsmath,amssymb,amsthm}
\usepackage{paralist}
\usepackage{mathrsfs}
\usepackage{epsfig}
\usepackage{epstopdf}
\usepackage[colorlinks=true]{hyperref}
\hypersetup{urlcolor=blue,citecolor=red}
\allowdisplaybreaks
\journal{Nonlinear Analysis: Real World Applications}

\newtheorem{theorem}{Theorem}[section]

\newtheorem{lemma}[theorem]{Lemma}
\newtheorem{proposition}[theorem]{Proposition}

\theoremstyle{definition}
\newtheorem{definition}[theorem]{Definition}
\newtheorem{remark}[theorem]{Remark}
\newtheorem{assumption}[theorem]{Assumption}

\newcommand{\tu}{\widetilde{u}}
\newcommand{\tw}{\widetilde{w}}
\newcommand{\tr}{\widetilde{\rho}}
\newcommand{\norm}[1]{\left\lVert #1\right\rVert}
\newcommand{\HH}{\mathcal{H}}
\newcommand{\X}{E}
\newcommand{\E}{\mathcal{E}}
\newcommand{\Fw}{\mathcal{E}_{\omega}}

\begin{document}
\begin{frontmatter}

\title{Continuous Data Assimilation for Memristive FitzHugh--Nagumo Neural Networks}
\author{Jing Tian}
\ead{jtian@towson.edu}
\address{Department of Mathematics, Towson University, Towson, Maryland 21252, USA}

\begin{abstract}
In this study, we investigate continuous data assimilation for a
class of partly diffusive FitzHugh–Nagumo neural networks with memristive
feedback and linear synaptic coupling. The data assimilation algorithm implemented
here only adds a feedback term to the membrane-potential equations, while the
recovery and memductance variables remain unobserved. We obtain weighted energy estimates and give sufficient conditions for exponential convergence of the assimilated solution to the reference solution under noise-free observations, covering full membrane-potential observations, general coarse observation operators, and orthogonal projections. We also derive an error estimate for locally square-integrable observation noise and show that, when the noise energy is bounded, the long-time error is bounded by a constant multiple of the noise-energy bound.
\end{abstract}


\end{frontmatter}

\section{Introduction}
FitzHugh--Nagumo equations are standard reduced models of excitable dynamics
\cite{ermentrout,fitzhugh,izhikevich}. In partly diffusive neural networks,
the membrane potentials satisfy reaction--diffusion equations, while the
recovery and memductance variables satisfy ordinary differential equations.
The memristive interaction is motivated by the theory and physical
realization of memristive devices \cite{chua,chuakang,itoh,strukov}.
Dissipative dynamics and synchronization have been studied for memristive
FitzHugh--Nagumo and related reaction--diffusion neural networks~\cite{you,youtu,youHR}.
In this paper, we study the recovery of an unknown trajectory of a
memristive FitzHugh--Nagumo network from partial observations via continuous
data assimilation.

Let a fully coupled network consist of $m\ge2$ neurons. For $1\le i\le m$,
the state of neuron $i$ is denoted by $(u_i,w_i,\rho_i)$,
where $u_i$ is the membrane potential, or voltage, $w_i$ is the recovery
variable, and $\rho_i$ is the memductance variable. Following \cite{you},
we consider the reference system
\begin{equation}
\begin{aligned}
\partial_t u_i&=\eta\Delta u_i+f(u_i,x)-\sigma w_i+J-k\tanh(\rho_i)\,u_i
+P\sum_{j=1}^m(u_j-u_i),\\
\partial_t w_i&=au_i+c-bw_i,\\
\partial_t \rho_i&=qu_i-r\rho_i,
\end{aligned}
\label{eq:reference}
\end{equation}
for $t>0$ and $x\in\Omega$, where $\Omega\subset\mathbb R^n$, $1\le n\le3$, is a bounded Lipschitz domain. The homogeneous Neumann boundary condition and the initial conditions are given by
\begin{equation}
\begin{aligned}
\frac{\partial u_i}{\partial\mathbf n}&=0\quad\text{on }\partial\Omega,\\
(u_i,w_i,\rho_i)(0,x)&=(u_i^0,w_i^0,\rho_i^0)(x),\qquad 1\le i\le m.
\end{aligned}
\label{eq:initialboundary}
\end{equation}
Here $\mathbf n$ denotes the outward unit normal. The reference initial state $(u_i^0,w_i^0,\rho_i^0)(x)$ is unknown and is
assumed to belong to the energy space. The assimilation algorithm and the
sufficient conditions on its nudging strength and spatial resolution will
not depend on this initial state.
The parameters satisfy
\begin{equation*}
\eta,\sigma,a,b,q,r>0,\qquad k,P\ge0,\qquad J,c\in\mathbb R.
\end{equation*}
Here $P$ is the linear synaptic coupling strength and $k$ is the coefficient of the memristive interaction. When \(k=0\), the memductance variable no longer feeds back into the voltage equation. The convergence estimate for this case will be given in Lemma~\ref{lem:kzero}.

Let $g(t)=\operatorname{col}(u_i,w_i,\rho_i:1\le i\le m)$ denote a reference
solution of \eqref{eq:reference}, where $\operatorname{col}$ denotes column stacking.
Its initial state is $g^0=\operatorname{col}(u_i^0,w_i^0,\rho_i^0:1\le i\le m)$.
To recover this trajectory from membrane-potential observations, we introduce
the assimilated state
$\widetilde g=\operatorname{col}(\tu_i,\tw_i,\tr_i:1\le i\le m)$,
which satisfies
\begin{equation}
\begin{aligned}
\partial_t \tu_i&=\eta\Delta\tu_i+f(\tu_i,x)-\sigma\tw_i+J-k\tanh(\tr_i)\,\tu_i\\
&\quad+P\sum_{j=1}^m(\tu_j-\tu_i)-\mu I_h(\tu_i-u_i),\\
\partial_t \tw_i&=a\tu_i+c-b\tw_i,\\
\partial_t \tr_i&=q\tu_i-r\tr_i,
\end{aligned}
\label{eq:nudged}
\end{equation}
with the same Neumann condition on $\tu_i$ and arbitrary initial data
\[
\widetilde g^0=\operatorname{col}(\tu_i^0,\tw_i^0,\tr_i^0:1\le i\le m).
\]
Here $\mu>0$
is the nudging strength. The operator $I_h$ is either the identity, for full membrane-potential observations, or a coarse observation operator with resolution $h$.

The algorithm considered here follows the nudging framework of Azouani,
Olson, and Titi~\cite{aot}, in which general interpolant observables are
used to recover an unknown trajectory. Continuous and discrete data
assimilation for dissipative systems has been studied in
\cite{biswas,biswasmartinez,du,farhatbenard,farhatns,hayden,larios}, while limitations for
certain non-dissipative equations without finitely many determining modes
are discussed in \cite{titivictor}. Related studies have addressed synchronization and state estimation in neural networks and reaction--diffusion systems~\cite{collin, ouannas,wei}.
Continuous data assimilation has also been considered in recent preprints
for semilinear parabolic equations, including bidomain models with
FitzHugh--Nagumo kinetics~\cite{delsarto}, and for learned surrogate
dynamics~\cite{lisanz}.

In this work, we first consider full membrane-potential observations ($I_h=I$)
to establish recovery of the unobserved variables, and then study the general coarse
observations
satisfying \eqref{eq:observationconditions}. When $I_h$ is an $L^2$-orthogonal
projection, we obtain a sharper estimate by using orthogonality.
Synchronization in this paper means convergence of the assimilated state
to the reference state, in contrast to synchronization among neurons
within the networks studied in \cite{you,youtu}. We prove
that, for noise-free observations, this convergence is exponential in the
energy space whenever the nudging strength and spatial resolution satisfy
suitable conditions. For $k>0$, the estimate starts after a transient period
depending on the reference trajectory. When $k=0$, the memristive feedback vanishes. Under the corresponding nudging conditions, we obtain an exponential error estimate valid for all $t\ge0$.

The main difficulty in the error estimates comes from the product of
the memductance error, the reference voltage, and the voltage error.
We control this term by combining memductance damping, an absorbing
\(L^4\) bound for the reference voltage, and interpolation for the voltage
error. We also prove uniqueness for \(L^2\) initial data in dimensions
\(1\le n\le3\) by using the \(H^1\) regularity
of the memductance difference between solutions with the same initial
conditions. Finally, for observation noise with bounded energy,
we obtain a long-time error bound proportional to the noise-energy bound.

The rest of the paper is organized as follows. Section~2 provides the
formulation and preliminaries. In Section~3, we prove the existence and
uniqueness for the reference system and obtain some estimates. In Section~4, we study the
nudged system and the memristive error term. In Section~5,
we prove exponential convergence for full observations, orthogonal
projections, and general coarse observations, and then derive the
noisy-observation estimates. Section~6 contains the concluding remarks.

\section{Formulation and preliminaries}\label{sec:preliminaries}
We write $L^p=L^p(\Omega)$, $1\le p\le\infty$, for the usual
Lebesgue spaces and $H^1=H^1(\Omega)$ for the usual Sobolev space.
The $L^2$ norm and inner product are denoted by $\norm{\cdot}$
and $(\cdot,\cdot)$, respectively, and $\norm{\cdot}_{L^p}$
denotes the $L^p$ norm for $p\ne2$.
We denote by $|\Omega|$ the Lebesgue measure of $\Omega$
and by $\langle\cdot,\cdot\rangle$ the duality pairing.
The space $H^1(\Omega)^*$ denotes the continuous dual of $H^1(\Omega)$.
For an open time interval $\mathcal I$, $\mathcal D'(\mathcal I)$
denotes the space of distributions on $\mathcal I$.

Let
\[
\X=[L^2(\Omega;\mathbb R^3)]^m
\]
be the energy space. The voltage component \(u_i\) is the only diffusive component and is therefore
controlled in \(H^1(\Omega)\); the recovery and
memductance components \(w_i\) and \(\rho_i\) are controlled in \(L^2(\Omega)\).

The energy norm is
\[
\norm{g}_{\X}^2=\sum_{i=1}^m(\norm{u_i}^2+\norm{w_i}^2+\norm{\rho_i}^2).
\]
For the weighted energy estimates below, fix the constants
\[
\omega:=\frac{\sigma}{a},\qquad c_{\min}:=\min\{1,\omega\},\qquad c_{\max}:=\max\{1,\omega\}.
\]

We impose the dissipativity and one-sided Lipschitz conditions of \cite{you},
together with a cubic growth condition on the reaction function $f$.
\begin{assumption}\label{ass:f}
The function $f\in C^1(\mathbb R\times\Omega)$ satisfies
\begin{equation}
f(s,x)s\le-\lambda|s|^4+\varphi(x),\qquad
\partial_sf(s,x)\le\beta,\qquad s\in\mathbb R,\ x\in\Omega,\label{eq:f}
\end{equation}
where $\lambda,\beta>0$ and $\varphi\in L^2(\Omega)$.
For the direct construction of weak solutions,
we assume the growth condition
\begin{equation}
|f(s,x)|\le C_f(1+|s|^3)+|\psi(x)|,\qquad C_f>0,\quad \psi\in L^2(\Omega).\label{eq:growth}
\end{equation}
\end{assumption}

\begin{definition}\label{def:weak}
For $T>0$, an energy solution on $[0,T]$ is a state $g$ such that, for each $i$,
\begin{align*}
u_i&\in C([0,T];L^2)\cap L^2(0,T;H^1)\cap L^4(0,T;L^4),\\
w_i,\rho_i&\in C([0,T];L^2),
\end{align*}
the initial conditions are satisfied in $\X$, and the following identities hold in
the sense of distributions in time. For every $v\in H^1(\Omega)$ and
$z\in L^2(\Omega)$,
\begin{align*}
\frac{d}{dt}(u_i,v)
&=-\eta(\nabla u_i,\nabla v)
+\langle f(u_i,\cdot),v\rangle
-\sigma(w_i,v)+J(1,v) \\
&\quad
-k(\tanh(\rho_i)u_i,v)
+P\sum_{j=1}^m(u_j-u_i,v),\\
\frac{d}{dt}(w_i,z)&=(au_i+c-bw_i,z),\\
\frac{d}{dt}(\rho_i,z)&=(qu_i-r\rho_i,z).
\end{align*}
The reaction pairing is well defined by \eqref{eq:growth}
and the Sobolev embedding $H^1(\Omega)\hookrightarrow L^4(\Omega)$.
An energy solution of the nudged system has the same regularity and
initial-value condition, with an extra term of
\(-\mu(I_h(\widetilde u_i-u_i),v)\).
\end{definition}

The recovery and memductance equations can be integrated explicitly.
Every energy solution therefore satisfies
\begin{equation}
\begin{aligned}
w_i(t)
&=e^{-bt}w_i^0
  +a\int_0^t e^{-b(t-s)}u_i(s)\,ds
  +\frac cb(1-e^{-bt}),
\\
\rho_i(t)
&=e^{-rt}\rho_i^0
  +q\int_0^t e^{-r(t-s)}u_i(s)\,ds.
\end{aligned}
\label{eq:hiddenformulas}
\end{equation}

For $n\le3$, the Gagliardo--Nirenberg inequality on
$\Omega$ gives a constant $C^*=C^*(\Omega,n)>0$ such that
\begin{equation}
\norm{v}_{L^4}^2
\le C^*\bigl(\norm{v}^2+
\norm{\nabla v}^{3/2}\norm{v}^{1/2}\bigr),
\qquad v\in H^1(\Omega).
\label{eq:gn}
\end{equation}
The lower-order term $\norm{v}^2$ is retained because the boundary
condition is homogeneous Neumann and nonzero constant functions have
zero gradient.

For $A,X,Y\ge0$ and $\theta>0$, Young's inequality gives
\begin{equation}
AX^{3/2}Y^{1/2}
\le \theta X^2+C_{Y}\theta^{-3}A^4Y^2.
\label{eq:young}
\end{equation}
Here $C_Y:=27/256$, obtained by applying Young's inequality with
conjugate exponents $4/3$ and $4$.
We also use the Young inequality
\begin{equation*}
xy\le \varepsilon x^2+\frac{y^2}{4\varepsilon},
\qquad \varepsilon>0.
\end{equation*}

Finally, we specify the observation operators in the assimilation
algorithm. Full membrane-potential observations correspond to $I_h=I$.
For coarse observations, we assume that the linear operator
$I_h:L^2(\Omega)\to L^2(\Omega)$ satisfies
\begin{equation}
\begin{aligned}
\norm{v-I_hv}
&\le c_I\,h\,\norm{\nabla v},
&&v\in H^1(\Omega),\\
\norm{I_hv}
&\le c_0\norm{v},
&&v\in L^2(\Omega),
\end{aligned}
\label{eq:observationconditions}
\end{equation}
where $h>0$ is the observation resolution and $c_I,c_0>0$ are
resolution-independent constants.

A useful special case is an $L^2$-orthogonal projection $I_h$
satisfying \eqref{eq:observationconditions}, with range
$V_h\subset L^2(\Omega)$. By definition, for every
$v\in L^2(\Omega)$,
\begin{equation*}
I_hv\in V_h,\qquad (v-I_hv,w)=0\quad\text{for all }w\in V_h.
\end{equation*}
 Examples include Neumann low-mode projections and
cell-average projections on cells with a uniform Poincar\'e constant.

\section{Estimates for the reference system}
In this section, we establish global existence, uniqueness, and absorbing estimates for the reference system. Related global-existence and absorbing estimates are given in \cite{you}. Here we work with energy solutions defined in Definition~\ref{def:weak} and use a different weighted energy to obtain the explicit bounds needed for the data assimilation analysis.

Set
\begin{equation*}
\HH(g)=\sum_i(\norm{u_i}^2+\omega\norm{w_i}^2+\norm{\rho_i}^2).
\end{equation*}
Thus $c_{\min}\norm{g}_{\X}^2\le\HH(g)\le c_{\max}\norm{g}_{\X}^2$. Introduce the constants
\begin{align*}
\nu&=\min\{1,b,r\},\qquad d=k+\frac12+\frac{q^2}{2r},\\
B_0&=m\left[\norm{\varphi}_{L^1}+\left(\frac{J^2}{2}+\frac{\omega c^2}{2b}\right)|\Omega|\right],
\quad C_0=B_0+\frac{m|\Omega|}{4\lambda}\left(d+\frac\nu2\right)^2.
\end{align*}
These constants are independent of the $g^0$ and $t$.
For $z\in\mathbb R$, completing the square gives
\begin{equation}
\lambda z^4-dz^2
\ge\frac\nu2z^2-\frac{(d+\nu/2)^2}{4\lambda}.
\label{eq:square}
\end{equation}

\begin{lemma}\label{lem:galerkin-limit}
Suppose that Assumption~\ref{ass:f} holds. For every initial state
$g^0\in\X$, the reference system \eqref{eq:reference}, with the boundary
and initial conditions \eqref{eq:initialboundary}, admits a global energy
solution in the sense of Definition~\ref{def:weak}.
\end{lemma}
\begin{proof}
Let $P_N$ be the $L^2$ projection onto the first $N$ Neumann
 eigenfunctions and denote the Galerkin
 approximations by
\[
g_N=\operatorname{col}(u_{i,N},w_{i,N},\rho_{i,N}:1\le i\le m),
\qquad g_N(0)=P_Ng^0,
\]
where $P_N$ acts componentwise.
By \eqref{eq:growth} and the ODE theory, the projected system admits local solutions.

Following the energy argument in the proof of Theorem~3.1 in \cite{you},
but using the weights $(1,\omega,1)$, we test the projected equations
with $(u_{i,N},\omega w_{i,N},\rho_{i,N})$ and sum over the neurons.
The network coupling is nonpositive, and we obtain
\[
\begin{aligned}
\frac12\frac{d}{dt}\HH(g_N)
&+\eta\sum_i\|\nabla u_{i,N}\|^2
+\frac{\omega b}{2}\sum_i\|w_{i,N}\|^2
+\frac r2\sum_i\|\rho_{i,N}\|^2
+\lambda\sum_i\|u_{i,N}\|_{L^4}^4\\
&\le d\sum_i\|u_{i,N}\|^2+B_0.
\end{aligned}
\]
Using \eqref{eq:square}, we first bound $g_N$ in
$L^\infty(0,T;\X)$ uniformly in $N$. Integrating the preceding energy inequality with the quartic term retained, we obtain
\[
\sup_{0\le t\le T}\|g_N(t)\|_{\X}^2
+\sum_{i=1}^m\int_0^T
\bigl(\|u_{i,N}\|_{H^1}^2+\|u_{i,N}\|_{L^4}^4\bigr)\,dt
\le C_T,
\]
where $C_T$ is independent of $N$. These bounds exclude finite-time blowup, so the Galerkin solutions
extend globally in time.

By \eqref{eq:growth}, the embedding $H^1\hookrightarrow L^4$,
and the $H^1$ stability of $P_N$,
\[
\|f(u_{i,N},\cdot)\|_{L^{4/3}((0,T)\times\Omega)}
+\|\partial_tu_{i,N}\|_{L^{4/3}(0,T;H^1(\Omega)^*)}
\le C_T,
\]
where $C_T$ is independent of $N$. Combining these estimates with the preceding energy bounds,
the Aubin--Lions theorem yields, after passing to a subsequence,
\[
u_{i,N}\to u_i\quad\text{strongly in }L^2((0,T)\times\Omega)
\quad\text{and almost everywhere}.
\]
Next, solving the projected recovery and memductance equations gives
\[
\begin{aligned}
w_{i,N}(t)&=e^{-bt}P_Nw_i^0
+a\int_0^te^{-b(t-s)}u_{i,N}(s)\,ds
+\frac cb(1-e^{-bt})P_N1,\\
\rho_{i,N}(t)&=e^{-rt}P_N\rho_i^0
+q\int_0^te^{-r(t-s)}u_{i,N}(s)\,ds.
\end{aligned}
\]
Therefore
\[
(w_{i,N},\rho_{i,N})\to(w_i,\rho_i)
\quad\text{strongly in }C([0,T];L^2(\Omega)^2).
\]
After a further subsequence, $\rho_{i,N}\to\rho_i$ almost everywhere.
Consequently,
\[
\begin{aligned}
&\|\tanh(\rho_{i,N})u_{i,N}-\tanh(\rho_i)u_i\|_{L^2((0,T)\times\Omega)}\\
&\quad\le \|u_{i,N}-u_i\|_{L^2((0,T)\times\Omega)}\\
&\qquad+\|(\tanh(\rho_{i,N})-\tanh(\rho_i))u_i\|_{L^2((0,T)\times\Omega)}
\longrightarrow0.
\end{aligned}
\]
The almost-everywhere convergence and uniform $L^{4/3}$ bound give
\[
f(u_{i,N},x)\rightharpoonup f(u_i,x)
\quad\text{in }L^{4/3}((0,T)\times\Omega).
\]
Since $P_Nv\to v$ in $L^4(\Omega)$ for every $v\in H^1(\Omega)$,
the reaction term passes to the limit in the Galerkin identities.

\smallskip

Testing the projected voltage equation with $\varphi(t)P_Nv$,
where $v\in H^1(\Omega)$, $\varphi\in C^\infty([0,T])$,
and $\varphi(T)=0$, and using $P_Nu_i^0\to u_i^0$ in $L^2$,
we pass to the limit to obtain
\[
(u_i(t),v)=(u_i^0,v)+\int_0^t\langle\partial_su_i(s),v\rangle\,ds,
\qquad v\in H^1(\Omega),
\]
for almost every $t$. Since
\[
u_i\in L^2(0,T;H^1)\cap L^4(0,T;L^4),\qquad
\partial_tu_i\in L^2(0,T;(H^1)^*)
+L^{4/3}(0,T;L^{4/3}),
\]
the energy chain rule, together with the preceding weak integral
identity, gives the energy identity and $L^2$-continuity \cite{gyongysiska}.
We first construct the solution on the larger interval $[0,T+\varepsilon]$, where $\varepsilon>0$, and then restrict it to $[0,T]$. We obtain
\[
u_i\in C([0,T];L^2),\qquad u_i(0)=u_i^0,
\]
and testing by $u_i$ is justified.
The integral formulas give
$w_i,\rho_i\in C([0,T];L^2)$ with their prescribed initial values.
A diagonal extraction yields a global solution.
\end{proof}

\begin{remark}\label{rem:forced-extension}
Lemma~\ref{lem:galerkin-limit} also holds when
$-\mu I_hu_i+F_i(t)$ is added to each voltage equation, where
$\mu\ge0$, $I_h:L^2(\Omega)\to L^2(\Omega)$ is bounded and linear,
and $F_i\in L^2_{\mathrm{loc}}([0,\infty);L^2(\Omega))$.
Young's inequality gives analogous finite-time energy bounds.
The additional Galerkin terms converge strongly in
$L^2(0,T;L^2)$, so the same argument yields a global solution with the regularity
of Definition~\ref{def:weak}.
\end{remark}

\begin{theorem}\label{thm:diss}
Suppose that Assumption~\ref{ass:f} holds. For every $g^0\in \X$, there is a global energy solution of \eqref{eq:reference} with boundary and initial conditions \eqref{eq:initialboundary}. Every such solution satisfies the energy inequality
\begin{equation}
\HH(g(t))\le e^{-\nu t}\HH(g^0)+\frac{2C_0}{\nu}(1-e^{-\nu t}),\qquad t\ge0.\label{eq:globalH}
\end{equation}
In particular, define $\log^+s=\max\{0,\log s\}$ for $s>0$ and
$\log^+0=0$. With
\begin{equation*}
H_*=1+\frac{2C_0}{\nu},\qquad K=\frac{H_*}{c_{\min}},\qquad
T_2(g^0)=\frac1\nu\log^+\HH(g^0),
\end{equation*}
one has $\norm{g(t)}_{\X}^2\le K$ for $t\ge T_2(g^0)$.
\end{theorem}
\begin{proof}
Existence follows from Lemma~\ref{lem:galerkin-limit}.
Following the proofs of Theorems~3.1 and 3.2 in \cite{you}, testing the three equations in \eqref{eq:reference} with
$u_i$, $\omega w_i$, and $\rho_i$, respectively, and summing
over $i$, we find that the voltage--recovery terms cancel
because $\omega a=\sigma$. Moreover,
\begin{equation}
P\sum_{i,j}(u_j-u_i,u_i)
=-\frac P2\sum_{i,j}\norm{u_i-u_j}^2\le0. \label{eq:graphref}
\end{equation}
By Assumption~\ref{ass:f} and Young's inequality, we obtain
\[
\begin{aligned}
\frac12\HH(g)'&+\eta\sum_i\norm{\nabla u_i}^2
+\frac{\omega b}{2}\sum_i\norm{w_i}^2\\
&+\frac r2\sum_i\norm{\rho_i}^2
+\lambda\sum_i\norm{u_i}_{L^4}^4
\le d\sum_i\norm{u_i}^2+B_0.
\end{aligned}
\]
Substituting \eqref{eq:square} and using $\nu\le b,r$, we obtain
\[
\HH(g)'+2\eta\sum_i\norm{\nabla u_i}^2+\nu\HH(g)\le2C_0.
\]
Integration gives \eqref{eq:globalH}.
For $t\ge T_2(g^0)$, we have
\[
c_{\min}\norm{g(t)}_{\X}^2
\le \HH(g(t))
\le 1+\frac{2C_0}{\nu}=H_*,
\]
and therefore $\norm{g(t)}_{\X}^2\le K$.
\end{proof}

\begin{proposition}\label{prop:refunique}
Suppose that Assumption~\ref{ass:f} holds. Let $g^{(1)}(t)$ and $g^{(2)}(t)$
be reference energy solutions on $[0,T]$, $T<\infty$, with
$g^{(1)}(0)=g^{(2)}(0)\in \X$. Then $g^{(1)}(t)=g^{(2)}(t)$ on $[0,T]$
in dimensions $1\le n\le3$.
\end{proposition}
\begin{proof}
Set
\[
\delta u_i=u_i^{(1)}-u_i^{(2)},\qquad
\delta w_i=w_i^{(1)}-w_i^{(2)},\qquad
\delta\rho_i=\rho_i^{(1)}-\rho_i^{(2)}.
\]
Applying \eqref{eq:hiddenformulas} and using the fact that the initial conditions are the same, we have
\[
\delta\rho_i(t)=q\int_0^t e^{-r(t-s)}\delta u_i(s)\,ds.
\]
Since $\delta u_i\in L^2(0,T;H^1(\Omega))$,
the integral formula implies
\[
\delta\rho_i\in W^{1,2}(0,T;H^1(\Omega))\subset C([0,T];H^1(\Omega)),
\qquad \delta\rho_i(0)=0\quad\text{in }H^1(\Omega).
\]
Choose $0<\chi\le\eta r/q^2$ and define
\begin{align*}
G(t)&=\sum_i\norm{\nabla \delta\rho_i(t)}^2,\\
\mathcal D_\chi(t)&=\sum_i\bigl(\norm{\delta u_i(t)}^2+
\omega\norm{\delta w_i(t)}^2+\norm{\delta\rho_i(t)}^2\bigr)+\chi G(t).
\end{align*}
Test the difference equations with $(\delta u_i,\omega\delta w_i,
\delta\rho_i)$ and sum over $i$. The voltage--recovery terms cancel, and the coupling contribution is nonpositive by \eqref{eq:graphref}. Moreover,
\[
\begin{aligned}
\langle f(u_i^{(1)},x)-f(u_i^{(2)},x),\delta u_i\rangle&\le\beta\norm{\delta u_i}^2,\\
q|(\delta u_i,\delta\rho_i)|&\le\frac r2\norm{\delta\rho_i}^2+\frac{q^2}{2r}\norm{\delta u_i}^2.
\end{aligned}
\]
The memristive term satisfies
\[
\begin{aligned}
&-k(\tanh(\rho_i^{(1)})u_i^{(1)}
-\tanh(\rho_i^{(2)})u_i^{(2)},\delta u_i)\\
&\qquad\le k\norm{\delta u_i}^2+k\int_\Omega|\delta\rho_i u_i^{(1)}\delta u_i|\,dx.
\end{aligned}
\]
For every $n\le3$, the embedding $H^1(\Omega)\hookrightarrow L^6(\Omega)$
and H\"older's inequality give
\[
\begin{aligned}
k\int_\Omega|\delta\rho_i u_i^{(1)}\delta u_i|\,dx
&\le k\norm{\delta\rho_i}_{L^6}\norm{u_i^{(1)}}_{L^3}\norm{\delta u_i}\\
&\le C_{\chi,k,\Omega}\norm{u_i^{(1)}}_{L^3}\mathcal D_\chi(t).
\end{aligned}
\]
Moreover, taking the spatial weak gradient of
$\partial_t\delta\rho_i=q\delta u_i-r\delta\rho_i$,
testing with $\chi\nabla\delta\rho_i$, summing over $i$, and applying Young's inequality,
we obtain 
\[
\frac{\chi}{2}G'+\chi rG
=\chi q\sum_i(\nabla\delta u_i,\nabla\delta\rho_i)
\le\frac{\chi r}{2}G
+\frac{\chi q^2}{2r}\sum_i\norm{\nabla\delta u_i}^2.
\]
Combining these estimates and using $\chi q^2/(2r)\le\eta/2$ yields
\begin{equation}
\begin{aligned}
\frac12\mathcal D_\chi'
&+\frac\eta2\sum_i\norm{\nabla \delta u_i}^2
+\omega b\sum_i\norm{\delta w_i}^2
+\frac r2\sum_i\norm{\delta\rho_i}^2+\frac{\chi r}{2}G\\
&\le C\left(1+\sum_i\norm{u_i^{(1)}(t)}_{L^3}\right)
\mathcal D_\chi(t).
\end{aligned}
\label{eq:uniquenessenergy}
\end{equation}
Here $C>0$ depends only on the fixed model parameters, $\chi$, and $\Omega$.
Since $u_i^{(1)}\in L^4(0,T;L^4)$ and $|\Omega|<\infty$,
\[
1+\sum_i\norm{u_i^{(1)}}_{L^3}\in L^1(0,T),\qquad
\mathcal D_\chi(0)=0.
\]
Gronwall's inequality in \eqref{eq:uniquenessenergy} gives
$\mathcal D_\chi\equiv0$, hence $g^{(1)}=g^{(2)}$ on $[0,T]$.
\end{proof}

Let $\lambda>0$ be the constant in Assumption~\ref{ass:f}, and define $C_\lambda=\frac23(\frac{8}{3\lambda})^{1/2}$,
\begin{align*}
D_0&=m\left[C_\lambda\norm{\varphi}_{L^{3/2}(\Omega)}^{3/2}
+\left(\frac{2J^2}{\lambda}+\frac{256k^3}{27\lambda^2}\right)|\Omega|\right],\\
Q&=\frac{4\sigma^2}{\lambda^2}K+\frac{2D_0}{\lambda}
+\frac{4}{27}m|\Omega|.
\end{align*}
These expressions are finite because $\varphi\in L^2(\Omega)$ and $|\Omega|<\infty$. 
\begin{theorem}\label{thm:L4}
Suppose that Assumption~\ref{ass:f} holds. For any reference energy solution in the sense of Definition~\ref{def:weak}, choose a time $s\ge T_2(g^0)$ at which $\sum_i\norm{u_i(s)}_{L^4}^4<\infty$. Then
\begin{equation}
\begin{aligned}
\sum_i\norm{u_i(t)}_{L^4}^4
&\le e^{-2\lambda(t-s)}\sum_i\norm{u_i(s)}_{L^4}^4\\
&\quad+Q(1-e^{-2\lambda(t-s)}),\qquad t\ge s.
\end{aligned}
\label{eq:F4bound}
\end{equation}
In particular, $\limsup_{t\to\infty}\Bigl(\sum_i\norm{u_i(t)}_{L^4}^4\Bigr)\le Q$.
\end{theorem}
\begin{proof}
By Lemma~\ref{lem:galerkin-limit},
$u_i\in L^4(0,T;L^4(\Omega))$ for every finite $T$,
so almost every $s>T_2(g^0)$ satisfies the stated condition. Fixing any such $s\ge T_2(g^0)$,
we use an $L^4$ energy estimate similar to the proof of Theorem~3.3 in
\cite{you}.
For energy solutions and $L>0$, use the bounded truncations
\[
\begin{aligned}
T_L(z)&=\max\{-L,\min\{z,L\}\},\quad
\Phi_L(z)=\int_0^zT_L(y)^3\,dy,\\
D_L(z)&=|z|^3|T_L(z)|^3.
\end{aligned}
\]
Since $|T_L(z)|\le |z|$, Young's inequality gives
\[
\begin{aligned}
& (\sigma|w_i|+|J|+k|u_i|)|T_L(u_i)|^3+|\varphi|\,|T_L(u_i)|^2
\\
&\le \frac\lambda2D_L(u_i)+\frac{2\sigma^2}{\lambda}|w_i|^2+\frac{2J^2}{\lambda}+C_\lambda|\varphi|^{3/2}
+\frac{256k^3}{27\lambda^2}.
\end{aligned}
\]
For fixed $L$, set $\beta_L(z)=T_L^3(z)$. Then
\[
|\beta_L|\le L^3,\qquad
\beta_L'(z)=3z^2\mathbf1_{\{|z|<L\}},\qquad
0\le\beta_L'\le3L^2\quad\text{a.e.}
\]
Smooth $u_i$, the reaction term $f(u_i,x)$, and the remaining
terms in the voltage equation in time by convolution
with the same smooth kernel. Denote the time-smoothed voltage by
$u_{i,\epsilon}$.
For every $0<t_a<t_b<T$, time convolution gives
\[
u_{i,\epsilon}\to u_i
\quad\text{in }L^2(t_a,t_b;H^1)\cap L^4(t_a,t_b;L^4).
\]
Since $\beta_L$ is Lipschitz,
\[
\|\beta_L(u_{i,\epsilon})-\beta_L(u_i)\|_{L^p(t_a,t_b;L^p)}
\le 3L^2\|u_{i,\epsilon}-u_i\|_{L^p(t_a,t_b;L^p)}
\longrightarrow 0,\qquad p=2,4.
\]
By the Sobolev chain rule,
\[
\begin{aligned}
&\|\nabla\beta_L(u_{i,\epsilon})
       -\nabla\beta_L(u_i)\|_{L^2(t_a,t_b;L^2)}\\
&\le 3L^2\|\nabla u_{i,\epsilon}-\nabla u_i\|_{L^2(t_a,t_b;L^2)}
 +\|[\beta_L'(u_{i,\epsilon})-\beta_L'(u_i)]
       \nabla u_i\|_{L^2(t_a,t_b;L^2)}.
\end{aligned}
\]
The first term tends to zero. Along a subsequence,
$u_{i,\epsilon}\to u_i$ a.e.; since $\beta_L'$ is continuous
away from $\pm L$ and $\nabla u_i=0$ a.e. on
$\{u_i=\pm L\}$, dominated convergence gives the same
conclusion for the second term, using
\[
\big|[\beta_L'(u_{i,\epsilon})-\beta_L'(u_i)]\nabla u_i\big|^2
\le36L^4|\nabla u_i|^2\in L^1((t_a,t_b)\times\Omega).
\]
Consequently, along this subsequence,
\[
\beta_L(u_{i,\epsilon})\to\beta_L(u_i)
\quad\text{in }L^2(t_a,t_b;H^1)\cap L^4(t_a,t_b;L^4),
\qquad 0<t_a<t_b<T.
\]
Passing to the limit in the two dual pairings yields
\[
\int_\Omega\Phi_L(u_i(t_b))\,dx
-\int_\Omega\Phi_L(u_i(t_a))\,dx
=\int_{t_a}^{t_b}\langle\partial_\tau u_i,
\beta_L(u_i)\rangle\,d\tau.
\]
Here the pairing combines the action of the diffusion and remaining terms on $H^1(\Omega)$ with the integral pairing of $f(u_i,x)$ and $\beta_L(u_i)$.
Since $\Phi_L$ is Lipschitz and $u_i\in C([0,T];L^2)$,
this identity extends to all $0\le t_a\le t_b\le T$.
We apply this identity with $t_a=s$ and $t_b=t$.
The diffusion and coupling terms satisfy
\[
\eta(\nabla u_i,\nabla\beta_L(u_i))
=3\eta\int_{\{|u_i|<L\}}u_i^2|\nabla u_i|^2\,dx\ge0,
\]
\[
\begin{aligned}
P\sum_{i,j}(u_j-u_i,\beta_L(u_i))
&=-\frac P2\sum_{i,j}
(u_i-u_j,\beta_L(u_i)-\beta_L(u_j))\\
&\le0,
\end{aligned}
\]
where the last inequality follows from the monotonicity
of $\beta_L$. Moreover,
\[
f(z,x)T_L(z)^3\le-\lambda D_L(z)+|\varphi(x)|\,|T_L(z)|^2.
\]
Combining these estimates gives
\[
\begin{aligned}
&\sum_i\int_\Omega\Phi_L(u_i(t))\,dx
+\frac{\lambda}{2}\int_s^t\sum_i\int_\Omega
D_L(u_i(\tau))\,dx\,d\tau \\
&\qquad\le
\sum_i\int_\Omega\Phi_L(u_i(s))\,dx
+\left(\frac{2\sigma^2}{\lambda}K+D_0\right)(t-s).
\end{aligned}
\]
As $L\to\infty$, $\Phi_L(z)\uparrow|z|^4/4$ and
$D_L(z)\uparrow|z|^6$. By monotone convergence, for
$Y(t)=\sum_i\norm{u_i(t)}_{L^4}^4$ we obtain
\[
\frac14Y(t)+\frac\lambda2\int_s^t\sum_i\norm{u_i(\tau)}_{L^6}^6\,d\tau
\le\frac14Y(s)+\left(\frac{2\sigma^2}{\lambda}K+D_0\right)(t-s).
\]
Using $|z|^6\ge |z|^4-4/27$, we have
\[
Y(t)+2\lambda\int_s^tY(\tau)\,d\tau\le Y(s)+2\lambda Q(t-s).
\]
The same estimate on subintervals with finite initial $L^4$ norm gives
$Y'+2\lambda Y\le2\lambda Q$ in $\mathcal D'(s,T)$.
The integrating factor yields \eqref{eq:F4bound} a.e.; $L^2$-continuity
and Fatou's lemma extend it to every $t\ge s$. Letting $t\to\infty$, we obtain $\limsup_{t\to\infty}\Bigl(\sum_i\norm{u_i(t)}_{L^4}^4\Bigr)\le Q$.
\end{proof}

\section{Estimates on the nudged system and the nonlinear memristive term}\label{sec:memristive-error}

\begin{proposition}\label{prop:exist}
Suppose that Assumption~\ref{ass:f} and \eqref{eq:observationconditions} hold, and fix a reference energy solution $g(t)$.
Then, for every $\widetilde g^0\in \X$, the nudged system
\eqref{eq:nudged} admits a unique global energy solution for $1\le n\le3$.
\end{proposition}
\begin{proof}
By \eqref{eq:observationconditions}, we have
\begin{align*}
-\mu(I_h(\tu_i-u_i),\tu_i)
&\le\mu c_0\norm{\tu_i}^2+\mu c_0\norm{u_i}\norm{\tu_i}\\
&\le\frac{3\mu c_0}{2}\norm{\tu_i}^2+\frac{\mu c_0}{2}\norm{u_i}^2.
\end{align*}
Taking the inner products with $(\tu_i,\omega\tw_i,\tr_i)$ and summing, we obtain
\begin{equation}\begin{aligned}
&\frac12\frac{d}{dt}\HH(\widetilde g)+\eta\sum_i\norm{\nabla\tu_i}^2
+\frac{\omega b}{2}\sum_i\norm{\tw_i}^2+\frac r2\sum_i\norm{\tr_i}^2
+\lambda\sum_i\norm{\tu_i}_{L^4}^4\\
&\qquad\le\left(d+\frac{3\mu c_0}{2}\right)\sum_i\norm{\tu_i}^2+B_0
+\frac{\mu c_0}{2}\sum_i\norm{u_i}^2.
\end{aligned}\label{eq:nudgedexistenergy}\end{equation}
Set $d_\mu=d+3\mu c_0/2$. By \eqref{eq:globalH},
$\sum_i\norm{u_i(t)}^2\le\HH(g^0)+2C_0/\nu$ for all $t\ge0$.
Thus \eqref{eq:nudgedexistenergy} and \eqref{eq:square} with $d$ replaced by $d_\mu$ give
\[
\HH(\widetilde g)'+\nu\HH(\widetilde g)\le B_\mu,
\]
where
\[
B_\mu=2B_0+\mu c_0\left(\HH(g^0)+\frac{2C_0}{\nu}\right)
+\frac{m|\Omega|}{2\lambda}(d_\mu+\nu/2)^2.
\]
Consequently,
\begin{equation}
\HH(\widetilde g(t))\le e^{-\nu t}\HH(\widetilde g^0)
+\frac{B_\mu}{\nu}(1-e^{-\nu t}),\qquad t\ge0.
\label{eq:nudgedfinite}
\end{equation}
For $t\ge T_2(g^0)$, use $\sum_i\norm{u_i}^2\le K$ instead to obtain
\begin{equation*}
\HH(\widetilde g)'+\nu\HH(\widetilde g)\le A_\mu,\qquad
A_\mu=2B_0+\mu c_0K+
\frac{m|\Omega|}{2\lambda}(d_\mu+\nu/2)^2,
\quad t\ge T_2(g^0).
\end{equation*}
Consequently,
\[
\limsup_{t\to\infty}\HH(\widetilde g(t))
\le\frac{A_\mu}{\nu}.
\]
For every $T<\infty$, the reference bound and \eqref{eq:observationconditions} give
$F_i=\mu I_hu_i\in L^2(0,T;L^2)$.
Passing to the Galerkin limit as in
Remark~\ref{rem:forced-extension} gives a solution of
\eqref{eq:nudged}. The bound \eqref{eq:nudgedfinite}
allows this solution to be continued for all $t\ge0$.

For two solutions with common initial data, set
\[
\delta u_i=\widetilde u_i^{(1)}-\widetilde u_i^{(2)},\quad
\delta w_i=\widetilde w_i^{(1)}-\widetilde w_i^{(2)},\quad
\delta\rho_i=\widetilde\rho_i^{(1)}-\widetilde\rho_i^{(2)}.
\]
The reference forcing cancels. Relative to Proposition~\ref{prop:refunique},
the only extra term satisfies
\[
-\mu(I_h\delta u_i,\delta u_i)
\le\mu\norm{I_h\delta u_i}\norm{\delta u_i}
\le\mu c_0\norm{\delta u_i}^2.
\]
As in that proposition,
\[
\delta\rho_i\in W^{1,2}(0,T;H^1),\qquad
\delta\rho_i(0)=0\ \text{in }H^1.
\]
With the same $\mathcal D_\chi$ and $0<\chi\le\eta r/q^2$  defined in the proof of Proposition~\ref{prop:refunique}, we have
\[
\begin{aligned}
\mathcal D_\chi'
&\le C(t)\mathcal D_\chi+2\mu c_0\sum_i\norm{\delta u_i}^2\\
&\le(C(t)+2\mu c_0)\mathcal D_\chi,
\qquad C\in L^1(0,T)\quad(n\le3).
\end{aligned}
\]
Thus $\mathcal D_\chi(0)=0$ and Gronwall's inequality imply
$\mathcal D_\chi\equiv0$, proving uniqueness.
\end{proof}

Fix \(\varepsilon>0\).  By Theorem~\ref{thm:L4}, there exists
\(T_0=T_0(\varepsilon,g)\ge0\) such that
\begin{equation}
\sum_i \norm{u_i(t)}_{L^4}^4\le M_4^4,\qquad
 t\ge T_0(\varepsilon,g),\qquad
M_4^4=Q+\varepsilon .
\label{eq:M}
\end{equation}

Set
\[
U_i=\tu_i-u_i,\qquad W_i=\tw_i-w_i,\qquad R_i=\tr_i-\rho_i.
\]

\begin{lemma}\label{lem:mem}
Let $u_i\in L^4(\Omega)$, $U_i\in H^1(\Omega)$, and
$\rho_i,\tr_i\in L^2(\Omega)$. With the differences defined above,
let $M_4>0$ and $\theta>0$.
If $\norm{u_i}_{L^4}\le M_4$, then
\begin{equation}\begin{aligned}
&-k(\tanh(\tr_i)\,\tu_i-\tanh(\rho_i)\,u_i,U_i)
\le{}\frac{r}{4}\norm{R_i}^2
+\theta\norm{\nabla U_i}^2\\
&\quad+\left(k+\frac{C^*k^2M_4^2}{r}
+C_Y\theta^{-3}\frac{(C^*)^4k^8M_4^8}{r^4}\right)
\norm{U_i}^2.
\end{aligned}\label{eq:memlemma}\end{equation}
For energy solutions satisfying \eqref{eq:M}, this estimate holds for
almost every $t\ge T_0$.
\end{lemma}
\begin{proof}
Following (4.4)-(4.11) in \cite{you}, write
\[
\tanh(\tr_i)\tu_i-\tanh(\rho_i)u_i
=\tanh(\tr_i)U_i+\bigl(\tanh(\tr_i)-\tanh(\rho_i)\bigr)u_i.
\]
Since $|\tanh|\le1$ and $\operatorname{Lip}(\tanh)\le1$, H\"older and Young give
\begin{align*}
&-k(\tanh(\tr_i)\tu_i-\tanh(\rho_i)u_i,U_i)\\
&\quad\le k\norm{U_i}^2+k\int_\Omega|R_i u_i U_i|\,dx\\
&\quad\le k\norm{U_i}^2+\frac r4\norm{R_i}^2
+\frac{k^2M_4^2}{r}\norm{U_i}_{L^4}^2.
\end{align*}
By \eqref{eq:gn} and the explicit Young inequality \eqref{eq:young},
\begin{align*}
\frac{k^2M_4^2}{r}\norm{U_i}_{L^4}^2
&\le \frac{C^*k^2M_4^2}{r}\norm{U_i}^2
+\frac{C^*k^2M_4^2}{r}\norm{\nabla U_i}^{3/2}\norm{U_i}^{1/2}\\
&\le\theta\norm{\nabla U_i}^2
+\left(\frac{C^*k^2M_4^2}{r}
+C_Y\theta^{-3}\frac{(C^*)^4k^8M_4^8}{r^4}\right)\norm{U_i}^2.
\end{align*}
Combining these estimates proves \eqref{eq:memlemma}.
On every finite interval $[T_0,T]$, the energy regularity and \eqref{eq:M} give
\[
R_i\in L^\infty(T_0,T;L^2),\quad
u_i\in L^\infty(T_0,T;L^4),\quad
U_i\in L^2(T_0,T;L^4).
\]
Hence, by H\"older,
\[
\int_{T_0}^T\!\!\int_\Omega|R_i u_i U_i|\,dx\,dt
\le (T-T_0)^{1/2}
\|R_i\|_{L^\infty L^2}\|u_i\|_{L^\infty L^4}\|U_i\|_{L^2L^4}<\infty,
\]
justifying time integration of \eqref{eq:memlemma}.
\end{proof}

\section{Exponential convergence of the assimilation algorithm}

Fix $M_4,T_0$ for the reference trajectory $g$ as in
Section~\ref{sec:memristive-error}. Decay estimates start at $T_0$, except
when $k=0$, which requires no $L^4$ bound. For $1\le n\le3$, both energy
solutions are unique by Propositions~\ref{prop:refunique} and~\ref{prop:exist}.

For the difference functions introduced above, define the unweighted and weighted error energies by
\[
\E(t)=\sum_i(\norm{U_i}^2+\norm{W_i}^2+\norm{R_i}^2),\qquad
\Fw(t)=\sum_i(\norm{U_i}^2+\omega\norm{W_i}^2+\norm{R_i}^2).
\]
These energies satisfy
\begin{equation}
c_{\min}\E(t)\le\Fw(t)\le c_{\max}\E(t),\qquad C_\omega=c_{\max}/c_{\min}.
\label{eq:equivalence}
\end{equation}
The initial factor in the estimates below satisfies
\[
\E(T_0)\le\frac{2}{c_{\min}}
\bigl(\HH(g(T_0))+\HH(\widetilde g(T_0))\bigr),
\]
which is controlled by the initial data and $T_0$ through
\eqref{eq:globalH} and \eqref{eq:nudgedfinite} for noise-free observations.

Subtracting \eqref{eq:reference} from \eqref{eq:nudged} gives
\begin{equation*}
\begin{aligned}
\partial_t U_i&=\eta\Delta U_i+f(\tu_i,x)-f(u_i,x)-\sigma W_i
-k(\tanh(\tr_i)\,\tu_i-\tanh(\rho_i)\,u_i)\\
&\quad+P\sum_{j=1}^m(U_j-U_i)-\mu I_hU_i,\\
\partial_t W_i&=aU_i-bW_i,\\
\partial_t R_i&=qU_i-rR_i,
\end{aligned}
\end{equation*}
with homogeneous Neumann condition for $U_i$; the shared forcing constants $J,c$ cancel.

For $k>0$, the subsequent error inequalities hold in $\mathcal D'(T_0,\infty)$. Indeed,
\[
\bigl\langle f(\widetilde u_i,x)-f(u_i,x),U_i\bigr\rangle
\le \beta\|U_i\|^2.
\]
Lemma~\ref{lem:mem} gives time integrability of the memristive term. By continuity, the integrated estimates hold for all $t\ge T_0$.

\subsection{Full membrane-potential observations}
For $I_h=I$, let
\begin{equation*}
\Lambda=\beta+k+\frac{q^2}{r}+\frac{C^*k^2M_4^2}{r}
+C_Y\frac{(C^*)^4k^8M_4^8}{\eta^3r^4}.
\end{equation*}
\begin{theorem}\label{thm:fullobs}
Let the reference and assimilated states be energy solutions in the sense of
Definition~\ref{def:weak}, under the standing hypotheses of Section~\ref{sec:preliminaries}.
Use $M_4$ and $T_0$ as in \eqref{eq:M}. If $I_h=I$ and
\begin{equation*}
\mu>\Lambda,
\end{equation*}
then, for all $t\ge T_0$,
\begin{equation}
\E(t)\le
C_\omega e^{-\gamma(t-T_0)}\E(T_0),
\qquad
\gamma=\min\{2b,r,2(\mu-\Lambda)\}>0.
\label{eq:fullrate}
\end{equation}
Thus the assimilated state converges to the reference state in $\X$,
although feedback acts only on the voltage equations.
\end{theorem}
\begin{proof}
Following the proof of Theorem~4.2 in \cite{you}, test with
$(U_i,\omega W_i,R_i)$. The voltage--recovery cross terms cancel
by $\omega a=\sigma$, and the coupling is nonpositive by \eqref{eq:graphref}.
By Young's inequality,
\begin{equation}
q(U_i,R_i)\le\frac r4\norm{R_i}^2+\frac{q^2}{r}\norm{U_i}^2.\label{eq:qerror}
\end{equation}
Using \eqref{eq:f} and Lemma~\ref{lem:mem} with $\theta=\eta$, we obtain
\begin{equation}
\frac12\Fw'+\omega b\sum_i\norm{W_i}^2+\frac r2\sum_i\norm{R_i}^2
+(\mu-\Lambda)\sum_i\norm{U_i}^2\le0,\qquad t\ge T_0.\label{eq:fullcollect}
\end{equation}
Since $\mu>\Lambda$, the definition of $\gamma$ yields $\gamma>0$ and $\Fw'+\gamma\Fw\le0$. 
Integration and \eqref{eq:equivalence} yield \eqref{eq:fullrate}.
\end{proof}

\begin{remark}\label{rem:fixed-weight}
The term $q^2/r$ in $\Lambda$ reflects the unit weight of $\|R_i\|^2$
in $\Fw$; for $k=0$, Lemma~\ref{lem:kzero} gives a less restrictive
condition. With fixed $\varepsilon$ and $\omega=\sigma/a$,
$M_4^4=O(m)$ and $\Lambda=O(m^2)$ as $m\to\infty$ because
Lemma~\ref{lem:mem} bounds each $\|u_i\|_{L^4}$ using
$\sum_i\|u_i\|_{L^4}^4\le M_4^4$. Sharper per-neuron absorbing
estimates could improve this dependence.
\end{remark}
\subsection{Coarse observations}
We next consider observation operators satisfying \eqref{eq:observationconditions}, first for orthogonal projections and then for general interpolants.

Set
\begin{equation*}
\Lambda_c=\Lambda+7C_Y\frac{(C^*)^4k^8M_4^8}{\eta^3r^4}.
\end{equation*}
\begin{theorem}\label{thm:coarseorth}
Assume the solution hypotheses of Theorem~\ref{thm:fullobs}, the conditions in \eqref{eq:observationconditions}, and that $I_h$ is an orthogonal projection on $L^2(\Omega)$. If
\begin{equation}
\mu>\Lambda_c,\qquad
c_I^2\,h^2<\frac{\eta}{2\Lambda_c},
\label{eq:orthconditions}
\end{equation}
then, with
\begin{equation*}
\delta_h^\perp=\min\left\{\mu,\frac{\eta}{2c_I^2\,h^2}\right\}
-\Lambda_c>0,
\qquad\gamma_h^\perp=\min\{2b,r,2\delta_h^\perp\},
\end{equation*}
\begin{equation}
\E(t)\le C_\omega e^{-\gamma_h^\perp(t-T_0)}\E(T_0),
\qquad t\ge T_0.
\label{eq:orthunweighted}
\end{equation}
The strict inequalities in \eqref{eq:orthconditions} are used only to ensure
\(\delta_h^\perp>0\); at equality the argument gives no exponential decay rate.
\end{theorem}
\begin{proof}
By orthogonality, we have
\[
(I_hU_i,U_i)=\norm{I_hU_i}^2,\qquad
\norm{U_i}^2=\norm{I_hU_i}^2+\norm{(I-I_h)U_i}^2.
\]
The approximation estimate implies
\begin{align}
\frac\eta2\norm{\nabla U_i}^2+\mu\norm{I_hU_i}^2
&\ge \frac{\eta}{2c_I^2h^2}\norm{(I-I_h)U_i}^2
+\mu\norm{I_hU_i}^2\notag\\
&\ge\min\left\{\mu,\frac{\eta}{2c_I^2h^2}\right\}\norm{U_i}^2.
\label{eq:orthcoercivity}
\end{align}
Combining Lemma~\ref{lem:mem} with $\theta=\eta/2$ and \eqref{eq:qerror}, we obtain
\[
\frac12\Fw'+\frac\eta2\sum_i\norm{\nabla U_i}^2
+\mu\sum_i\norm{I_hU_i}^2+\omega b\sum_i\norm{W_i}^2
+\frac r2\sum_i\norm{R_i}^2
\le\Lambda_c\sum_i\norm{U_i}^2.
\]
By \eqref{eq:orthcoercivity},
\begin{equation}
\frac12\Fw'
+\delta_h^\perp\sum_i\norm{U_i}^2
+\omega b\sum_i\norm{W_i}^2
+\frac r2\sum_i\norm{R_i}^2\le0.
\label{eq:orthcollect}
\end{equation}
Thus $\Fw'+\gamma_h^\perp\Fw\le0$. Integration, together with \eqref{eq:equivalence}, gives \eqref{eq:orthunweighted}.
\end{proof}

For general linear $I_h$, diffusion controls both interpolation error and the memristive term. Define
\begin{equation*}
\delta_h=\mu-\Lambda_c-
\frac{\mu^2c_I^2\,h^2}{2\eta}.
\end{equation*}
\begin{theorem}\label{thm:coarse}
Assume the solution hypotheses of Theorem~\ref{thm:fullobs} and \eqref{eq:observationconditions}. If
\begin{equation*}\delta_h>0,\end{equation*}
then
\begin{equation}
\E(t)\le C_\omega e^{-\gamma_h(t-T_0)}\E(T_0),\qquad t\ge T_0,
\qquad
\gamma_h=\min\{2b,r,2\delta_h\}>0.
\label{eq:coarseunweighted}
\end{equation}
A sufficient pair of conditions is
\begin{equation*}
\mu>2\Lambda_c,\qquad
\mu c_I^2\,h^2\le\eta.
\end{equation*}
\end{theorem}
\begin{proof}
The feedback term satisfies
\begin{align*}
-\mu(I_hU_i,U_i)&=-\mu\norm{U_i}^2+\mu(U_i-I_hU_i,U_i)\\
&\le-\mu\norm{U_i}^2+\mu c_I\,h\,\norm{\nabla U_i}\norm{U_i}\\
&\le-\mu\norm{U_i}^2+\frac\eta2\norm{\nabla U_i}^2
+\frac{\mu^2c_I^2\,h^2}{2\eta}\norm{U_i}^2.
\end{align*}
Using Lemma~\ref{lem:mem} with $\theta=\eta/2$ gives
\begin{equation}
\frac12\Fw'
+\omega b\sum_i\norm{W_i}^2
+\frac r2\sum_i\norm{R_i}^2
+\delta_h\sum_i\norm{U_i}^2\le0.
\label{eq:coarsecollect}
\end{equation}
It follows that $\Fw'+\gamma_h\Fw\le0$. Integrating and using \eqref{eq:equivalence}, we obtain \eqref{eq:coarseunweighted}. The sufficient conditions follow from
\begin{equation*}
\frac{\mu^2c_I^2\,h^2}{2\eta}\le\frac\mu2
\quad\Longrightarrow\quad
\delta_h\ge\frac\mu2-\Lambda_c>0.
\end{equation*}
\end{proof}

\subsection{Observational noise and the case \texorpdfstring{$k=0$}{k=0}}\label{sec:noise-kzero}
Let the observed data be
\begin{equation*}
y_i(t)=u_i(t)+\xi_i(t),\qquad\xi_i\in L^2(0,T;L^2(\Omega))\quad\text{for every }T>0,
\end{equation*}
and use $-\mu I_h(\tu_i-y_i)$ in the voltage equation. Since
\[
I_hy_i=I_hu_i+I_h\xi_i,
\]
the added terms are $\mu(I_h\xi_i,v)$ in the weak voltage identity
of Definition~\ref{def:weak} and $\mu I_h\xi_i$ in the error equation.
By \eqref{eq:observationconditions},
\[
\begin{gathered}
F_i=\mu I_h(u_i+\xi_i)\in L^2(0,T;L^2(\Omega))\quad(T>0),\\
\mu|(I_h\xi_i,\tu_i)|\le \tfrac12\norm{\tu_i}^2+
\tfrac{\mu^2c_0^2}{2}\norm{\xi_i}^2.
\end{gathered}
\]
Thus Remark~\ref{rem:forced-extension} applies. Uniqueness follows from
Proposition~\ref{prop:exist}: the noise cancels between solutions with the
same initial data. Define
\begin{equation*}
\mathcal Z(t):=\sum_{i=1}^m\norm{\xi_i(t)}^2.
\end{equation*}

\begin{theorem}\label{thm:noise}
Assume the solution hypotheses of Theorem~\ref{thm:fullobs}, with
$M_4$ and $T_0$ chosen as in \eqref{eq:M}, and use the noisy observations
defined above. For general coarse interpolants satisfying
\eqref{eq:observationconditions}, assume $\delta_h>0$ and set
$\Gamma_h=\min\{2b,r,\delta_h\}$. Then, for all $t\ge T_0$,
\begin{equation}
\Fw(t)
\le e^{-\Gamma_h(t-T_0)}\Fw(T_0)
+\frac{\mu^2c_0^2}{\delta_h}\int_{T_0}^t e^{-\Gamma_h(t-s)}\mathcal Z(s)\,ds.
\label{eq:noisybound}
\end{equation}
If $\mathcal Z(t)\le\mathcal Z_*$ for almost every sufficiently large $t$, then
\begin{equation}
\begin{aligned}
\limsup_{t\to\infty}\Fw(t)
&\le\frac{\mu^2c_0^2}{\Gamma_h\delta_h}\mathcal Z_*,\\
\limsup_{t\to\infty}\E(t)
&\le\frac{\mu^2c_0^2}
{c_{\min}\Gamma_h\delta_h}\mathcal Z_*.
\end{aligned}
\label{eq:noisefloor}
\end{equation}
The same bounds hold with $c_0=1$ and the replacements
\[
(\delta_h,\Gamma_h)\longmapsto
\begin{cases}
(\mu-\Lambda,\min\{2b,r,\mu-\Lambda\}),
& I_h=I,\quad\mu>\Lambda,\\
(\delta_h^\perp,\min\{2b,r,\delta_h^\perp\}),
& \begin{gathered}
I_h\text{ orthogonal},\quad\mu>\Lambda_c,\\
c_I^2\,h^2<\frac{\eta}{2\Lambda_c},
\end{gathered}
\end{cases}
\]
\end{theorem}
\begin{proof}
For any $\vartheta>0$, the noise term satisfies
\begin{equation}\label{eq:noise-young}
\mu|(I_h\xi_i,U_i)|
\le\mu c_0\norm{\xi_i}\norm{U_i}
\le\frac{\vartheta}{2}\norm{U_i}^2+\frac{\mu^2c_0^2}{2\vartheta}\norm{\xi_i}^2.
\end{equation}
Include the additional noise term in the derivation of \eqref{eq:coarsecollect}.
Apply \eqref{eq:noise-young} with
$\vartheta=\delta_h$ and multiply by two to obtain
\begin{equation*}
\Fw'
+2\omega b\sum_i\norm{W_i}^2
+r\sum_i\norm{R_i}^2
+\delta_h\sum_i\norm{U_i}^2
\le\frac{\mu^2c_0^2}{\delta_h}\mathcal Z(t).
\end{equation*}
Therefore,
\begin{equation*}
\Fw'+\Gamma_h\Fw
\le\frac{\mu^2c_0^2}{\delta_h}\mathcal Z(t).
\end{equation*}
An integrating factor gives \eqref{eq:noisybound}. If
$\mathcal Z(s)\le\mathcal Z_*$ a.e. for $s\ge S\ge T_0$, the convolution
over $[T_0,S]$ vanishes as $t\to\infty$, and that over $[S,t]$ is bounded
by $\mathcal Z_*/\Gamma_h$. Equation~\eqref{eq:equivalence} then gives
\eqref{eq:noisefloor}. Adding the noise term to \eqref{eq:fullcollect}
and \eqref{eq:orthcollect} proves the other cases.
\end{proof}

\begin{lemma}\label{lem:kzero}
Assume $k=0$. Set
\[
\mathcal E_{uw}(t)=\sum_i\bigl(\norm{U_i}^2+\omega\norm{W_i}^2\bigr),
\qquad
\mathcal E_\rho(t)=\sum_i\norm{R_i}^2,
\]
and define
\[
\delta_0:=
\begin{cases}
\mu-\beta, & I_h=I,\\[1mm]
\displaystyle \min\left\{\mu,\frac{\eta}{c_I^2h^2}\right\}-\beta,
& I_h \text{ is an orthogonal projection},\\[3mm]
\displaystyle \mu-\beta-\frac{\mu^2c_I^2h^2}{4\eta},
& I_h \text{ is a general interpolant}.
\end{cases}
\]
Assume $\delta_0>0$.
\smallskip

\noindent\textit{Noise-free observations.}
Let $\alpha_*=\min\{2b,2\delta_0\}$. Then
\[
\mathcal E_{uw}(t)\le e^{-\alpha_*t}\mathcal E_{uw}(0),
\qquad t\ge0.
\]
For every $0<\gamma<\min\{\alpha_*,2r\}$, there exists
$C_\gamma>0$, independent of time and initial errors, such that
\[
\E(t)\le C_\gamma e^{-\gamma t}\E(0),\qquad t\ge0.
\]

\smallskip
\noindent\textit{Noisy observations.}
Let
\[
\gamma_0=\min\{2b,\delta_0\},\qquad
\zeta=\frac{r\gamma_0}{2q^2},\qquad
\gamma_*=\min\{\gamma_0/2,r\},
\]
and $\mathcal E_\zeta=\mathcal E_{uw}+\zeta\mathcal E_\rho$.
Then, for $t\ge0$,
\[
\mathcal E_\zeta(t)
\le e^{-\gamma_*t}\mathcal E_\zeta(0)
+\frac{\mu^2c_0^2}{\delta_0}
\int_0^t e^{-\gamma_*(t-s)}\mathcal Z(s)\,ds.
\]
If $\mathcal Z(t)\le\mathcal Z_*$ for almost every sufficiently
large $t$, then
\[
\limsup_{t\to\infty}\E(t)
\le
\frac{\mu^2c_0^2\mathcal Z_*}
{\min\{1,\omega,\zeta\}\delta_0\gamma_*}.
\]
Take $c_0=1$ for full membrane-potential observations and orthogonal projections.
\end{lemma}

\begin{proof}
For $k=0$, the voltage-error equation is independent of $R_i$.
Test the $U_i,W_i$ equations by $U_i,\omega W_i$, sum in $i$, and use
$\omega a=\sigma$ to obtain
\[
\begin{aligned}
&\frac12\mathcal E_{uw}'
+\eta\sum_i\norm{\nabla U_i}^2
+\omega b\sum_i\norm{W_i}^2
+\mu\sum_i(I_hU_i,U_i)\\
&\qquad\le
\beta\sum_i\norm{U_i}^2
+\mu\sum_i(I_h\xi_i,U_i),
\end{aligned}
\]
where the nonpositive coupling contribution has been discarded.

For full observations, the feedback contributes $\mu\norm{U_i}^2$.
For orthogonal projections, orthogonality and the approximation
estimate give
\[
\begin{aligned}
\eta\norm{\nabla U_i}^2+\mu\norm{I_hU_i}^2
&\ge
\frac{\eta}{c_I^2h^2}\norm{(I-I_h)U_i}^2
+\mu\norm{I_hU_i}^2\\
&\ge
\min\left\{\mu,\frac{\eta}{c_I^2h^2}\right\}\norm{U_i}^2.
\end{aligned}
\]
For general interpolants, Young's inequality yields
\[
\mu|(U_i-I_hU_i,U_i)|
\le \mu c_Ih\norm{\nabla U_i}\norm{U_i}
\le \eta\norm{\nabla U_i}^2
+\frac{\mu^2c_I^2h^2}{4\eta}\norm{U_i}^2.
\]
Combining these estimates, we obtain in all three cases
\begin{equation}\label{eq:kzero-energy}
\mathcal E_{uw}'
+2\omega b\sum_i\norm{W_i}^2
+2\delta_0\sum_i\norm{U_i}^2
\le 2\mu\sum_i|(I_h\xi_i,U_i)|.
\end{equation}

First consider noise-free observations. If $\xi_i=0$, comparison of the coefficients gives
\[
\mathcal E_{uw}'+\alpha_*\mathcal E_{uw}\le0,
\]
which proves the bound for $\mathcal E_{uw}$.
Since $\partial_tR_i=qU_i-rR_i$, we have
\[
\mathcal E_\rho(t)^{1/2}
\le e^{-rt}\mathcal E_\rho(0)^{1/2}
+q\int_0^t e^{-r(t-s)}\mathcal E_{uw}(s)^{1/2}\,ds.
\]
For $0<\gamma<\min\{2r,\alpha_*\}$,
\[
\begin{aligned}
\int_0^t e^{-r(t-s)}e^{-\alpha_*s/2}\,ds
&\le e^{-\gamma t/2}
\int_0^t e^{-(r-\gamma/2)(t-s)}\,ds\\
&\le \frac{e^{-\gamma t/2}}{r-\gamma/2}.
\end{aligned}
\]
Consequently,
\[
\mathcal E_\rho(t)^{1/2}
\le e^{-\gamma t/2}
\left(
\mathcal E_\rho(0)^{1/2}
+\frac{q\mathcal E_{uw}(0)^{1/2}}{r-\gamma/2}
\right).
\]
Squaring with $(a+b)^2\le2a^2+2b^2$ and using the bound for
$\mathcal E_{uw}$ and energy equivalence proves the estimate for $\E$.

Next consider noisy observations. Applying \eqref{eq:noise-young}
with $\vartheta=\delta_0$ to \eqref{eq:kzero-energy} gives
\[
\mathcal E_{uw}'+\gamma_0\mathcal E_{uw}
\le \frac{\mu^2c_0^2}{\delta_0}\mathcal Z(t).
\]
Also,
\[
\mathcal E_\rho'
=2q\sum_i(U_i,R_i)-2r\mathcal E_\rho
\le \frac{q^2}{r}\mathcal E_{uw}-r\mathcal E_\rho.
\]
Multiplying the latter inequality by $\zeta=r\gamma_0/(2q^2)$
and adding gives
\[
\mathcal E_\zeta'
+\frac{\gamma_0}{2}\mathcal E_{uw}
+r\zeta\mathcal E_\rho
\le \frac{\mu^2c_0^2}{\delta_0}\mathcal Z(t).
\]
Since $\gamma_*=\min\{\gamma_0/2,r\}$, it follows that
\[
\mathcal E_\zeta'+\gamma_*\mathcal E_\zeta
\le \frac{\mu^2c_0^2}{\delta_0}\mathcal Z(t).
\]
Integration gives the stated convolution bound.
If $\mathcal Z(t)\le\mathcal Z_*$ for almost every $t\ge T_*$
for some $T_*\ge0$, integrating from $T_*$ yields
\[
\mathcal E_\zeta(t)
\le e^{-\gamma_*(t-T_*)}\mathcal E_\zeta(T_*)
+\frac{\mu^2c_0^2\mathcal Z_*}{\delta_0\gamma_*}
\left(1-e^{-\gamma_*(t-T_*)}\right),
\qquad t\ge T_*.
\]
Taking the limit superior and using
$\mathcal E_\zeta\ge\min\{1,\omega,\zeta\}\E$
proves the final estimate.
\end{proof}

\section{Conclusion}
We have established a continuous data assimilation framework for the
memristive FitzHugh--Nagumo reaction--diffusion network using only
membrane-potential observations of all neurons. Under explicit sufficient
conditions, voltage nudging recovers the full state exponentially from
noise-free observations. The analysis covers full spatial observations,
coarse orthogonal projections, and general interpolants. For $k>0$, convergence is established after the
reference trajectory enters the required $L^4$ absorbing regime. When
$k=0$, the voltage error drives the memductance error without feedback,
allowing a less restrictive threshold and convergence estimates from
$t=0$ without an $L^4$ bound.

We have also studied the energy estimates with locally
square-integrable observation noise through a time integral of the squared noise norm with an exponentially decaying weight. If the noise energy is uniformly bounded at sufficiently large times,
these estimates give an explicit long-time bound on the error energy.
The sufficient thresholds are not claimed to be
sharp; in particular, sharper per-neuron absorbing estimates may improve
their dependence on network size. Numerical validation of these bounds
and extensions to discrete-time observations will be studied in the 
future.

\end{document}